\documentclass[11pt]{amsart}
\usepackage[english]{babel}
\usepackage{array,mathpazo,amsmath,xcolor}

\usepackage{stmaryrd, mathtools, enumerate, amssymb}

\usepackage[colorlinks=true,pagebackref]{hyperref}

\usepackage[all]{xy}
\SelectTips{eu}{} 
\xyoption{curve}

\def\lto{{\longrightarrow}}
\def\into{{\hookrightarrow}}
\def\xto{\xrightarrow}

\newcommand{\call}{{\mathcal L}}

\newcommand{\calo}{{\mathcal O}}

\newcommand{\PP}{{\mathbb P}}

\newcommand{\ZZ}{{\mathbb Z}}

\newcommand{\bfa}{{\mathbf a}}
\newcommand{\bfb}{{\mathbf b}}
\newcommand{\bfc}{{\mathbf c}}
\newcommand{\bfd}{{\mathbf d}}

\newcommand{\bfx}{{\mathbf x}}
\newcommand{\bfy}{{\mathbf y}}

\newcommand{\sfG}{\ensuremath{\mathsf G}}

\DeclareMathOperator{\adj}{adj}

\DeclareMathOperator{\car}{char}

\DeclareMathOperator{\coker}{coker}
\DeclareMathOperator{\cok}{coker}
\DeclareMathOperator{\Coker}{Coker}

\DeclareMathOperator{\Ext}{Ext}
\DeclareMathOperator{\Extgr}{Extgr}

\DeclareMathOperator{\GL}{GL}

\DeclareMathOperator{\id}{id}

\DeclareMathOperator{\Jac}{Jac}

\DeclareMathOperator{\Mat}{\mathsf{Mat}}

\DeclareMathOperator{\rank}{rank}

\DeclareMathOperator{\res}{res}

\DeclareMathOperator{\SL}{SL}

\DeclareMathOperator{\tr}{tr}

\newcommand{{\sbullet}}{{\scriptstyle\bullet}}

\newcommand{\diag}{\mathsf {diag}}

\swapnumbers
\theoremstyle{definition}
\newtheorem{defn}{Definition}[section]

\newtheorem{remark}[defn]{Remark}

\newtheorem{remarks}[defn]{Remarks}

\newtheorem{sit}[defn]{}
\newtheorem{example}[defn]{Example}

\theoremstyle{plain}

\newtheorem{proposition}[defn]{Proposition}
\newtheorem{theorem}[defn]{Theorem}
\newtheorem{theorema}[defn]{Theorem A}
\newtheorem{theoremb}[defn]{Theorem B}

\newtheorem{lemma}[defn]{Lemma}
\newtheorem{cor}[defn]{Corollary}

\let\oldmarginpar\marginpar
\def\marginpar#1{\oldmarginpar{\tiny #1}}

\begin{document}
\title[Ulrich Bundles]
{Moore Matrices and Ulrich Bundles on an Elliptic Curve}

\author{Ragnar-Olaf Buchweitz}
\address{Dept.\ of Computer and Math\-ematical Sciences,
University  of Tor\-onto at Scarborough, 
1265 Military Trail, 
Toronto, ON M1C 1A4,
Canada}
\email{ragnar@utsc.utoronto.ca}

\author{Alexander Pavlov}
\address{
Mathematical Sciences Research Institute (MSRI),
17 Gauss Way, Berkeley, CA 94720, USA}
\email{apavlov@msri.org}

\dedicatory{To Frank-Olaf Schreyer on the occasion of his sixtieth birthday}
\thanks{
This material is based upon work supported by the National Science Foundation under grant No.0932078 000 while the second author was in residence at the Mathematical Sciences Research Institute in Berkeley, California, during the 2015/2016 year. The first author was partly supported by NSERC Discovery Grant 3-642-114-80.}
\date{\today}



\begin{abstract} 
We give normal forms of determinantal representations of a smooth projective plane cubic in terms of
{\em Moore matrices}. Building on this, we exhibit matrix factorizations for all indecomposable vector bundles
of rank $2$ and degree $0$ without nonzero sections, also called {\em Ulrich bundles}, on such curves. 
\end{abstract}
\maketitle

{\footnotesize\tableofcontents}
%
%\section{Introduction}
%TBW
\section{Linear Determinantal Representations of Smooth Hesse Cubics and Statement of the Results}
\begin{sit}
Let $K$ be a field and $f\in S=K[x_{0},...,x_{n}]$ a nonzero homogeneous polynomial of degree $d\geqslant 1$.
A matrix $M=(m_{ij})_{i,j=0,...,d-1}\in\Mat_{d\times d}(S)$ is {\em linear\/} if its entries
are linear homogeneous polynomials and it provides a {\em linear determinantal representation\/} of 
$f$ if $\det M \doteq f$, where $\doteq$ means that the two quantities are equal up to a nonzero scalar from $K$. 
More generally, if $M$ is a linear matrix for which there exists a matrix $M'$ with 
$M\cdot M' = f\id_{d}= M'\cdot M$, then $M$ is an {\em Ulrich matrix\/} and $\mathbf F=\Coker M$, necessarily 
annihilated by $f$, is an {\em Ulrich module\/} over $R=S/(f)$. 

If $V(f)=\{f=0\}\subset \PP^{n}$, the projective hypersurface defined by $f$, is smooth, then the sheafification
of $\mathbf F$ is a vector bundle, called an {\em Ulrich bundle}. 

Our aim here is to give normal forms for all Ulrich bundles of rank $1$ or $2$ over a plane elliptic curve in 
Hesse form.
\end{sit}

Linear determinantal representations of hypersurfaces have been studied since, at least, the middle
of the $19^{th}$ century and for a very recent comprehensive treatment for curves and surfaces see
Dolgachev's monograph \cite{DolgBook}. This reference contains as well a detailed study of the geometry of
smooth cubic curves, especially of those in Hesse form and we refer to it for background material.

For smooth plane projective curves, the state-of-the-art result is due to Beauville.

\begin{theorem}[{\cite[Prop.3.1]{Beau00}}]
\label{beau}
Let $C=V(f)$ be a smooth plane projective curve of degree $d$ defined by an equation $f=0$ in $\PP^{2}$
over $K$. With $g=\tfrac{1}{2}(d-1)(d-2)$ the genus of $C$, one has the following.
\begin{enumerate}[\rm (a)]
\item
Let $L$ be a line bundle of degree $g-1$ on $C$ with $H^{0}(X, L)=0$. 
Then there exists a $d \times d$ linear matrix $M$ such that $f=\det M$ and an exact
sequence of $\calo_{\PP^{2}}$--modules
\begin{align*}
\xymatrix{
0\ar[r]&\calo_{\PP^{2}}(-2)^{d}\ar[r]^-{M}&\calo_{\PP^{2}}(-1)^{d}\ar[r]&L\ar[r]&0\,.
}
\end{align*}
\item
Conversely, let $M$ be a $d\times d$ linear matrix such that  $f\doteq \det M$. Then the
cokernel of $M: \calo_{\PP^{2}}(-2)^{d}\to \calo_{\PP^{2}}(-1)^{d}$ is a line bundle $L$ on $C$ of degree 
$g-1$ with $H^{0}(X, L)=0$.\qed
\end{enumerate}
\end{theorem}
In this result, the matrix $M$ clearly determines $L$ uniquely up to isomorphism, 
but $L$ determines $M$ only up to equivalence of matrices
in that the cokernel of every matrix $PMQ^{-1}$ for $P,Q\in \GL(d,K)$ yields a line 
bundle isomorphic to $L$.
The so obtained action on these matrices of the group 
$\sfG_{d}= \GL(d,K)\times \GL(d,K)/\mathbb G_{m}(K)$, with $\mathbb G_{m}(K)$ the diagonally embedded 
subgroup of nonzero multiples of the identity matrix, is free and proper; see {\cite[Prop.3.3]{Beau00}}. 
The geometric quotient by this group action identifies with the affine variety $\Jac^{g-1}(C)\setminus \Theta$, 
the Jacobian of $C$ of line bundles of degree $g-1$ minus the theta divisor $\Theta$ of those line 
bundles of that degree that have a nonzero section.

There is therefore the issue of finding useful representatives, or normal forms, for such linear 
representations in a given orbit. 
Realizing hyperelliptic curves as double covers of $\PP^{1}$, Mumford \cite[IIIa,\S 2]{Mum84} exhibited canonical 
presentations for such line bundles, which motivated Beauville's work  \cite{Beau90}, and Laza--Pfister--Popescu in 
\cite{LPP} found such representative matrices for the Fermat cubic, while for 
general elliptic curves in Weierstra\ss\ form Galinat \cite{Galinat14}  determines normal forms of those 
linear representations.

Our first result yields the following normal forms for plane elliptic curves in Hesse form. Here, and in the sequel,
we denote $[a_{0}{:}\cdots{:}a_{n}]\in \PP^{n}(K)$ the $K$--rational point with homogeneous coordinates $a_{i}\in K$, 
not all zero.

\begin{theorema}
Over an algebraically closed field%
\footnote{Likely, it suffices that $K$ contains six distinct sixth roots of unity, which forces  $\car K\neq 2,3$.
However, one key ingredient in the proof, see (\ref{action}) below, is stated in the literature only over algebraically closed fields,
thus, we are compelled to make that assumption too.}
 $K$ of characteristic $\car K\neq 2,3$, each linear determinantal 
representation of the smooth plane projective curve $E$ with equation
\[
x_{0}^{3}+x_{1}^{3}+x_{2}^{3}+\lambda x_{0}x_{1}x_{2}=0\,,\quad  \lambda\in K\,,\lambda^{3}+27\neq 0\,,
\]
is equivalent to a {\em Moore matrix\/} 
\[
M_{(a_{0},a_{1},a_{2}),\bfx} = (a_{i+j}x_{i-j})_{i,j\in\ZZ/3\ZZ} =
\left(
\begin{array}{ccc}
a_{0}x_{0}&a_{1}x_{2}&a_{2}x_{1}\\
a_{1}x_{1}&a_{2}x_{0}&a_{0}x_{2}\\
a_{2}x_{2}&a_{0}x_{1}&a_{1}x_{0}
\end{array}
\right)
\]
with $\bfa=[a_{0}{:}a_{1}{:}a_{2}]\in E$ and $a_{0}a_{1}a_{2}\neq 0$.

Two such Moore matrices $M_{(a_{0},a_{1},a_{2}),\bfx}$ and $M_{(a_{0}',a_{1}',a_{2}'),\bfx}$ 
yield equivalent linear determinantal representations of $E$ if, and only if, 
$3\cdot_{E}\bfa =3\cdot_{E}\bfa'$, where
$
3\cdot_{E}\bfa =\bfa +_{E}\bfa +_{E}\bfa
$
is calculated with respect to the group law $+_{E}$ on $E$ whose
identity element is the inflection point $\mathfrak o =[0{:}-1{:}1]$.
\end{theorema} 

\begin{remarks}
\begin{enumerate}[\rm(a)]
\item 
Choosing an inflection point as origin for the group law,
the exceptional points $\bfa\in E$ with $a_{0}a_{1}a_{2}=0$ are precisely the $3$--torsion points, equivalently, 
the inflection points of $E$. They form the subgroup 
\[
E[3]=\{ [1{:}-\omega{:} 0], [0{:}1{:}-\omega], [-\omega{:} 0{:} 1]\mid \omega\in K, \omega^{3}=1\}\subseteq E\,,
\]
isomorphic to the elementary abelian $3$--group $\ZZ/3\ZZ\times\ZZ/3\ZZ$ of rank $2$.
\item
In geometric terms, the preceding result states that the map $\bfa \mapsto \cok M_{a,x}$ from 
$E\setminus E[3]$ to its punctured Jacobian  $\Jac^{0}(E)\setminus \{[\calo_{E}]\}\cong E\setminus E[1]$ 
of line bundles of degree $0$ without nonzero sections is well defined and is 
(isomorphic to the restriction of) the isogeny%
\footnote{We thank Steve Kudla for suggesting this interpretation.} 
of degree $9$ that is given by multiplication by $3$ on $E$.
\end{enumerate}
\end{remarks}

Building on the previous result, our second contribution is as follows.
\begin{theoremb}
\label{thmb}
Let $E$ be the smooth plane cubic curve from above.
\begin{enumerate}[\rm(a)]
\item 
\label{thmba}
Let $F$ be an indecomposable vectorbundle of rank $2$ and degree $0$ on $E$.
If $H^{0}(E,F)=0$, then there exists $a=(a_{0}, a_{1}, a_{2})\in K^{3}$ representing a point $\bfa\in E$ with 
$a_{0}a_{1}a_{2}\neq 0$ such that the sequence of $\calo_{\PP^{2}}$--modules
\begin{align*}
\xymatrix{
0\ar[r]&\calo_{\PP^{2}}(-2)^{6}\ar[rrr]^-{
{\left(
\begin{array}{cc}
{M_{a, x}} &M_{b, x}\\
0&M_{a, x}
\end{array}\right)}
}&&&\calo_{\PP^{2}}(-1)^{6}\ar[r]&F\ar[r]&0\,,
}
\end{align*}
with $b=(b_{0},b_{1},b_{2})\in K^{3}$ representing $2\cdot_{E}\bfa$, is exact. 
\item
\label{thmbb}
Conversely, if $a$ represents $\bfa\in E$ with $a_{0}a_{1}a_{2}\neq 0$ and $b$ represents 
$2\cdot_{E}\bfa$, then the cokernel of the block matrix
\[\left(
\begin{array}{cc}
{M_{a,x}} &M_{b,x}\\
0&M_{a,x}
\end{array}\right):\calo_{\PP^{2}}(-2)^{6}\lto \calo_{\PP^{2}}(-1)^{6}
\]
is an indecomposable vectorbundle $F$ of rank $2$ and degree $0$ on $E$ that has no nonzero 
sections, $H^{0}(E,F)=0$.
\item
\label{thmbc} 
Replacing $a$ by $a'$ results in a vector bundle isomorphic to $F$ if, and only if,
$3\cdot_{E}\bfa = 3\cdot_{E}\bfa'$ on $E$.
\end{enumerate}
\end{theoremb}
In the next section we will review some known facts about elliptic curves in Hesse form and will prove 
Theorem A. In section 3 we will establish Theorem B.

\section{Proof of Theorem A}
To lead up to the proof of Theorem A, we first review some ingredients. To begin with, we review how
Moore matrices encode conveniently the group law on a smooth Hesse cubic $E$.

\subsection*{Moore matrices and their rank}  
\begin{defn}
With $a=(a_{0}, a_{1}, a_{2})\in K^{3}$, and $x=(x_{0},x_{1},x_{2})\in S=K[x_{0},x_{1},x_{2}]$ the vector of coordinate linear forms, the {\em Moore matrix\/} defined by $a$ is
\begin{align*}
M_{a,x} &= (a_{i+j}x_{i-j})_{i,j\in\ZZ/3\ZZ} =
\left(
\begin{array}{ccc}
a_{0}x_{0}&a_{1}x_{2}&a_{2}x_{1}\\
a_{1}x_{1}&a_{2}x_{0}&a_{0}x_{2}\\
a_{2}x_{2}&a_{0}x_{1}&a_{1}x_{0}
\end{array}
\right)
\intertext{with {\em adjugate}, or {\em signed cofactor matrix}}
M_{a,x}^{\adj} &=(a_{i+j-1} a_{i+j+1}x_{j-i}^{2}-a_{i+j}^{2}x_{j-i-1}x_{j-i+1})_{ i,j\in\ZZ/3\ZZ}\\
&=\left(
\begin{array}{ccc}
a_{1}a_{2}x_{0}^{2}-a_{0}^{2}x_{1}x_{2}&
a_{0}a_{2}x_{1}^{2}-a_{1}^{2}x_{0}x_{2}&
a_{0}a_{1}x_{2}^{2}-a_{2}^{2}x_{0}x_{1}\\
a_{0}a_{2}x_{2}^{2}-a_{1}^{2}x_{0}x_{1}&
a_{0}a_{1}x_{0}^{2}-a_{2}^{2}x_{1}x_{2}&
a_{1}a_{2}x_{1}^{2}-a_{0}^{2}x_{0}x_{2}\\
a_{0}a_{1}x_{1}^{2}-a_{2}^{2}x_{0}x_{2}&
a_{1}a_{2}x_{2}^{2}-a_{0}^{2}x_{0}x_{1}&
a_{0}a_{2}x_{0}^{2}-a_{1}^{2}x_{1}x_{2}
\end{array}
\right)
\end{align*}
so that
\begin{align*}
\det M_{a,x} = M_{a,x}M_{a,x}^{\adj}= M_{a,x}^{\adj}M_{a,x}=
a_{0}a_{1}a_{2}(x_{0}^{3}+x_{1}^{3}+x_{2}^{3}) - 
(a_{0}^{3}+a_{1}^{3}+a_{2}^{3})x_{0}x_{1}x_{2}\,.
\end{align*}
\end{defn}

If now $a_{0}a_{1}a_{2}\neq 0$, then set $\lambda = (a_{0}^{3}+a_{1}^{3}+a_{2}^{3})/a_{0}a_{1}a_{2}\in K$
to obtain 
\[\det M_{a,x} \doteq  f := x_{0}^{3}+x_{1}^{3}+x_{2}^{3} -\lambda x_{0}x_{1}x_{2}\,,
\] 
whence $M_{a, x}$ indeed yields a determinantal presentation of the cubic curve $C=V(f)$ and the point 
$\bfa=[a_{0}{:}a_{1}{:}a_{2}]$ in $\PP^{2}(K)$ underlying $a$ lies on $C$.

Note that $C$ will be smooth if, and only if, $\lambda^{3}\neq 27$ in $K$. In the smooth case
we write $E=V(f)$ to remind the reader that this curve is {\em elliptic} over $K$ in that it is smooth of 
genus $1$ and contains at least one point, for example $[0{:}-1{:}1]$, defined over $K$.

\begin{remark}
As pointed out in \cite{Pa15}, Moore matrices have already appeared in the literature in a variety of contexts: to 
describe equations of projective embeddings of elliptic curves; see \cite{GP98}; to give an explicit formula for the 
group operation on a cubic in Hesse form; see \cite{Frium02} and \cite{Ra}; as differential in a projective 
resolution of the field over elliptic algebras; see \cite{ATvdB90}. 

Although the matrices above were known, and are easily established to form a matrix factorization of a 
Hesse cubic cf. \cite[example $3.6.5$]{EG10}, their relation to line bundles of degree $0$ via representations of 
the Heisenberg group, as we establish below, seems first to have been observed in  \cite{Pa15}.
\end{remark}

The next result is well-known and easily established through, say, explicit calculation
as in \cite[Lemma 3]{Frium02}.

\begin{lemma}
\label{rank2}
For $E=V(f)$ a smooth cubic in Hesse form and every pair $a, b$ with $\bfa,\bfb \in E$,
the (specialized) Moore matrix $M_{a,b}$ is of rank $2$. 
\qed 
\end{lemma}
In the situation of the preceding Lemma, basic Linear Algebra tells us that the one-dimensional null space
$M_{a,b}^{\perp}=\{(c_{0},c_{1},c_{2})\in K^{3}\mid M_{a,b}\cdot (c_{0},c_{1},c_{2})^{T}=0\}$ is
spanned by the column vectors of $M_{a,b}^{\adj}$. 
Following Dolgachev, we denote 
\[
\mathfrak l(M_{a,b})=\bfc =[c_{0}{:} c_{1}{:} c_{2}]\in\PP^{2}(K)
\]
the point underlying the one-dimensional space $M_{a,b}^{\perp}$ in $\PP^{2}$. 

 A less obvious result, part of \cite[Thm.4]{Frium02}, is that $\bfc$ will again be a point of $E$ 
 along with $\bfa,\bfb$.

\begin{sit}
Before we turn to the group structure, let us note that the transpose of a Moore matrix is again a Moore matrix,
\begin{align*}
M_{a,b}^{T} =(a_{j+i}b_{j-i})_{i,j\in\ZZ/3\ZZ} = (a_{i+j}b_{-(i-j)})_{i,j\in\ZZ/3\ZZ} = M_{a,\iota(b)}\,,
\end{align*}
where $\iota$ is the involution $\iota(x_{0},x_{1},x_{2})= (x_{0},x_{2},x_{1})$, or, 
counting indices modulo $3$, $\iota(x_{j})=x_{-j}$. For use below, we follow again Dolgachev and set
\begin{align*}
\mathfrak r(M_{a,b}) &=\mathfrak l(M_{a,b}^{T}) = \mathfrak l(M_{a,\iota(b)})\in\PP^{2}(K)\,.
\end{align*}
The point $\mathfrak r(M_{a,b}) =\bfd = [d_{0}{:} d_{1}{:} d_{2}]$ thus represents the
space ${}^{\perp}M_{a,x}$ of solutions to the linear system of equations $d\cdot M_{a,b}=0$, 
spanned by the row vectors of $M_{a,b}^{\adj}$. 
\end{sit}

\begin{theorem}[see {\cite[Thm.4]{Frium02}}]
\label{addE}
The assignment $(\bfa,\bfb)\mapsto \bfc = \mathfrak l(M_{a,b})$ for $\bfa,\bfb\in E$ defines
the group law on $E$ by setting $\bfb-_{E}\bfa =\bfc$. The identity element is given by $\mathfrak o =[0{:}-1{:}1]$
and the inverse of $\bfa$ is $-_{E}\bfa = \iota(\bfa) = [a_{0}{:}a_{2}{:}a_{1}]$.\qed
\end{theorem}

Before continuing towards the proof of Theorem A, we take the opportunity to interpret Moore matrices 
geometrically in two ways, following Artin--Tate--van den Bergh \cite{ATvdB90} in the first and 
Dolgachev \cite{DolgBook} in the second. 

\subsection*{Geometric Interpretation  \`a la Artin--Tate--Van den Bergh}\quad\\
\begin{sit}
\label{ATV}
In the introduction to \cite{ATvdB90} the authors consider%
\footnote{Their indexing of the ingredients is different, but that is just due to the fact that the authors 
of  \cite{ATvdB90} chose the point at infinity $[1{:}-1{:}0]$ as the origin for the group law  on $E$.}
the trilinear forms
\begin{align}
f=\left(
\begin{array}{c}
f_{0}\\
f_{1}\\
f_{2}
\end{array}
\right) &= 
\left(
\begin{array}{c}
a_{0}x_{0}y_{0}+a_{1}x_{2}y_{1}+a_{2}x_{1}y_{2}\\
a_{1}x_{1}y_{0}+a_{2}x_{0}y_{1}+a_{0}x_{2}y_{2}\\
a_{2}x_{2}y_{0}+a_{0}x_{1}y_{1}+a_{1}x_{0}y_{2}
\end{array}
\right)
\end{align}
that can as well be interpreted as a system of three linear equations in, at least, two ways:
\begin{align}
\label{atv2}
\left(
\begin{array}{ccc}
a_{0}x_{0}&a_{1}x_{2}&a_{2}x_{1}\\
a_{1}x_{1}&a_{2}x_{0}&a_{0}x_{2}\\
a_{2}x_{2}&a_{0}x_{1}&a_{1}x_{0}
\end{array}
\right)
\left(
\begin{array}{c}
y_{0}\\
y_{1}\\
y_{2}
\end{array}
\right)&=
\left(
\begin{array}{c}
f_{0}\\
f_{1}\\
f_{2}
\end{array}
\right) 
=
\left[\left(
\begin{array}{ccc}
x_{0}&
x_{1}&
x_{2}
\end{array}
\right)
\left(
\begin{array}{ccc}
a_{0}y_{0}&a_{2}y_{1}&a_{1}y_{2}\\
a_{2}y_{2}&a_{1}y_{0}&a_{0}y_{1}\\
a_{1}y_{1}&a_{0}y_{2}&a_{2}y_{0}
\end{array}
\right)
\right]^{T}\,,
\intertext{or, shorter, in terms of Moore matrices,}
M_{a,x}\cdot y^{T} &= f = \big(x \cdot M_{\iota(a),\iota(\bfy)}\big)^{T} 
=M_{\iota(a),y}\cdot x^{T}\,.
\end{align}
\end{sit}

\begin{sit}
Viewing, for a fixed $\bfa\in E$, the $f_{i}$ as sections of $\calo(1,1)$ on 
$\PP^{2}_{x}\times \PP^{2}_{\bfy}$, these equations 
imply, by Lemma (\ref{rank2}), that the subscheme 
$X=V(f_{0}, f_{1}, f_{2}) \subseteq  \PP^{2}_{\bfx}\times \PP^{2}_{\bfy}$ is mapped
isomorphically by each of the projections $p_{\bfx},p_{\bfy}:\PP^{2}_{x}\times \PP^{2}_{\bfy}\lto \PP^{2}$
onto $E\subseteq \PP^{2}$, and, in light of the preceding Theorem, the subscheme $X$ constitutes 
the graph of the translation by $-\bfa$ on the elliptic curve in that
\begin{align*}
t_{-\bfa}&=p_{\bfy}(p_{\bfx}|_{X})^{-1} : 
E\xto{\ \cong\ } E\,,\quad t_{-\bfa}(\bfx) =\mathfrak l(M_{a,x}) = \bfx -_{E}\bfa =\bfy
\intertext{when going from $\PP^{2}_{x}$ to  $\PP^{2}_{\bfy}$, while it represents the graph of the 
translation by $\bfa$,} 
t_{\bfa}&=p_{\bfx}(p_{\bfy}|_{X})^{-1}: 
E\xto{\ \cong\ } E\,,\quad t_{\bfa}(\bfx) =  \mathfrak l(M_{\iota(a),x}) = \bfx +_{E}\bfa =\bfy
\end{align*} 
when going in the opposite direction. In other words,
\begin{align*}
X&=\{(\bfx,\bfx-_{E}\bfa)\in \PP^{2}_{\bfx}\times \PP^{2}_{\bfy}\mid \bfx\in E\}
=\{(\bfy+_{E}\bfa,\bfy)\in \PP^{2}_{\bfx}\times \PP^{2}_{\bfy}\mid \bfy\in E\}\,.
\end{align*}
\end{sit}

\subsection*{Geometric Interpretation \`a la Dolgachev}
\begin{sit}
Applying the treatment from \cite[4.1.2]{DolgBook} to the special case of plane elliptic curves gives a geometric interpretation of the adjugate of a Moore matrix as follows.

Fixing again $\bfa\in E$, consider the closed embedding
\begin{gather*}
(\mathfrak l,\mathfrak r)_{\bfa}: E\into \PP^{2}\times \PP^{2}\,,\\
(\mathfrak l,\mathfrak r)_{\bfa}(\bfx) = (\mathfrak l(M_{a,x}),\mathfrak r(M_{a,x}))
= (\bfx -_{E} \bfa, -_{E}\bfx -_{E}\bfa)\,,
\end{gather*}
and follow it with the Segre embedding $s_{2}:\PP^{2}\times \PP^{2}\into \PP^{8}$ that sends
$(\bfx,\bfy)$ to the class of the $3\times 3$ matrix $[\bfx^{T}\cdot\bfy]$ viewed as a point in $\PP^{8}$.
The composition $\psi_{\bfa}=s_{2}(\mathfrak l,\mathfrak r)_{\bfa}: E\into \PP^{8}$ then sends 
$\bfx\mapsto \left[M_{\bfa,\bfx}^{\adj}\right] = (\bfx-_{E}\bfa)^{T}\cdot(-_{E}\bfx-_{E}\bfa)$, thus, $E$ gets
embedded into the Segre variety $s_{2}(\PP^{2}\times \PP^{2})\subset \PP^{8}$ through the adjugate of the
Moore matrix. As to the image of $E$ in $\PP^{2}\times \PP^{2}$, if we set $\bfy=\bfx -_{E}\bfa$, then
$\bfx=\bfy+_{E}\bfa$ and $-_{E}\bfx-_{E}\bfa = -_{E}\bfy-_{E}2\cdot_{E}\bfa$ so that 
$(\mathfrak l,\mathfrak r)_{\bfa}(E)$ is the graph of the involution%
\footnote{That Moore matrices define an involution on $E$ in this way we learned from Kristian Ranestad
who kindly shared his notes \cite{Ra} with us.}
$\bfy\mapsto -_{E}\bfy-_{E}2\cdot_{E}\bfa$
on $E$ that one may view as the ``reflection'' in $-_{E}\bfa$.
\end{sit}

\subsection*{Doubling and Tripling Points on $E$}
\begin{sit}
\label{double}
As an immediate application of Theorem \ref{addE} one can easily determine%
\footnote{The formulas for $2\cdot_{E}\bfa$ are already contained in \cite{Frium02}. 
We recall them here for completeness und later use.}
 $2\cdot_{E}\bfa$ and 
$3\cdot_{E}\bfa$ for $\bfa\in E$ in that $2\cdot_{E}\bfa = \mathfrak l(M_{\iota(a),a})$ and
$3\cdot_{E}\bfa =  \mathfrak l(M_{\iota(a),b})$, where $\bfb=2\cdot_{E}\bfa$, and explicit coordinates 
are obtained from the columns of the corresponding adjugate matrices. Now
\begin{align*}
M_{\iota(\bfa),\bfa} = \left(a_{-i-j}a_{i-j}\right)_{i,j\in\ZZ{/}3\ZZ} &=
\left(
\begin{array}{ccc}
a_{0}^{2}&a_{2}^{2}&a_{1}^{2}\\
a_{2}a_{1}&a_{1}a_{0}&a_{0}a_{2}\\
a_{1}a_{2}&a_{0}a_{1}&a_{2}a_{0}
\end{array}
\right)
\intertext{and so}
M_{\iota(\bfa),\bfa}^{\adj} = \left(a_{-i-j}a_{i-j}\right)^{\adj}_{i,j\in\ZZ{/}3\ZZ}
&=
\left(
\begin{array}{ccc}
a_{0}(a_{2}^{3}-a_{1}^{3})\vphantom{\big(}\\
a_{2}(a_{1}^{3}-a_{0}^{3})\vphantom{\Big(}\\
a_{1}(a_{0}^{3}-a_{2}^{3})\vphantom{\big(}
\end{array}
\right)\cdot (0,-1,1)\,,
\end{align*}
whence $2\cdot_{E}\bfa = [a_{0}(a_{2}^{3}-a_{1}^{3}){:} a_{2}(a_{1}^{3}-a_{0}^{3}){:} a_{1}(a_{0}^{3}-a_{2}^{3})]$.
\end{sit}

Next, set $(b_{0}, b_{1}, b_{2})= (a_{0}(a_{2}^{3}-a_{1}^{3}),a_{2}(a_{1}^{3}-a_{0}^{3}), a_{1}(a_{0}^{3}-a_{2}^{3}))$ and evaluate through a straightforward though somewhat lengthy expansion:
\begin{align*}
M_{\iota(a), b}^{\adj} &= \left(a_{-i-j}b_{i-j}\right)^{\adj}_{i,j\in\ZZ{/}3\ZZ} =
\left(
\begin{array}{ccc}
a_{1}a_{2}b_{0}^{2}-a_{0}^{2}b_{1}b_{2}&a_{0}a_{1}b_{1}^{2}-a_{2}^{2}b_{0}b_{2}&a_{0}a_{2}b_{2}^{2}-
a_{1}^{2}b_{0}b_{1}\vphantom{\Big(}\\
a_{0}a_{1}b_{2}^{2}-a_{2}^{2}b_{0}b_{1}&a_{0}a_{2}b_{0}^{2}-a_{1}^{2}b_{1}b_{2}&a_{1}a_{2}b_{1}^{2}-
a_{0}^{2}b_{0}b_{2}
\\
a_{0}a_{2}b_{1}^{2}-a_{1}^{2}b_{0}b_{2}&a_{1}a_{2}b_{2}^{2}-a_{0}^{2}b_{0}b_{1}&a_{0}a_{1}b_{0}^{2}-
a_{2}^{2}b_{1}b_{2}\vphantom{\Big(}
\end{array}
\right)\\
&=
\left(
\begin{array}{c}
a_{0}a_{1}a_{2}(a_{0}^{6}+a_{1}^{6} +a_{2}^{6} - a_{0}^{3}a_{1}^{3} - a_{1}^{3}a_{2}^{3} - a_{0}^{3}a_{2}^{3})
\vphantom{\Big(}\\
a_{0}^{6}a_{1}^{3} + a_{1}^{6}a_{2}^{3} + a_{2}^{6}a_{0}^{3}-3(a_{0}a_{1}a_{2})^{3}\\
a_{0}^{6}a_{2}^{3} + a_{1}^{6}a_{0}^{3} + a_{2}^{6}a_{1}^{3}-3(a_{0}a_{1}a_{2})^{3}
\vphantom{\Big(}
\end{array}
\right)\cdot
(a_{0}, a_{2},a_{1})\,. 
\end{align*}
For an additional check, note that $[a_{0}{:}a_{2}{:}a_{1}]=-_{E}\bfa = -2_{E}\bfa +_{E}\bfa$ so that indeed
$[M_{\iota(\bfa),\bfa}^{\adj}]= (2_{E}\bfa +_{E}\bfa)^{T}(-2_{E}\bfa +_{E}\bfa)=(3_{E}\bfa)^{T}(-_{E}\bfa)$, 
as it has to be.
 
When $a_{0}a_{1}a_{2}\neq 0$, the case we are interested in, these results can be simplified a bit.
\begin{cor}
\label{3fold}
For $a=(a_{0},a_{1},a_{2})$ as above representing a point $\bfa\in E$ with $a_{0}a_{1}a_{2}\neq 0$, doubling, 
respectively tripling $\bfa$ on $E$ results in
\begin{gather*}
2\cdot_{E}\bfa =
\left[
\tfrac{a_{2}^{3}-a_{1}^{3}}{a_{1}a_{2}}{:}
\tfrac{a_{1}^{3}-a_{0}^{3}}{a_{0}a_{1}}{:}
\tfrac{a_{0}^{3}-a_{2}^{3}}{a_{0}a_{2}}\right]\,,\\
3\cdot_{E}\bfa =
\left[
\tfrac{a_{0}^{6}+a_{1}^{6} +a_{2}^{6}}{(a_{0}a_{1}a_{2})^{2}} - \tfrac{a_{0}^{3}a_{1}^{3} + 
a_{1}^{3}a_{2}^{3} + a_{0}^{3}a_{2}^{3}}{(a_{0}a_{1}a_{2})^{2}}{:}
\tfrac{a_{0}^{6}a_{1}^{3} + a_{1}^{6}a_{2}^{3} + a_{2}^{6}a_{0}^{3}}{(a_{0}a_{1}a_{2})^{3}}-3{:}
\tfrac{a_{0}^{6}a_{2}^{3} + a_{1}^{6}a_{0}^{3} + a_{2}^{6}a_{1}^{3}}{(a_{0}a_{1}a_{2})^{3}}-3
\right]\,.
\end{gather*}\qed
\end{cor}

\begin{example}
As an immediate application, one obtains the set of $6$--torsion points on $E$ in that $2\cdot_{E}\bfa$
is a $3$--torsion point if, and only if, $\bfa$ is a $6$--torsion point. Now $E[3] = E\cap V(x_{)}x_{1}x_{2})$
as was noted above. The formulae for doubling a point thus show that
\[
E[6]= E\cap V(x_{0}x_{1}x_{2}(x_{0}^{3}-x_{2}^{3})(x_{1}^{3}-x_{3}^{3})(x_{2}^{3}-x_{1}^{3}))
\]
is the intersection of $E$ with the indicated $12$ lines. As the four lines $V(x_{0}x_{1}x_{2}(x_{2}-x_{1}))$
cut out the $3$-- and $2$--torsion points, the remaining $8$ lines cut out the $24$
primitive $6$--torsion points as stated in \cite{Frium02}. 

It follows that $E[6]\cong \ZZ/6\ZZ\times \ZZ/6\ZZ$
and that all $36$ points of $E[6]$ are defined over $K$, as soon as $\car K\neq 2,3$ and $K$ contains 
three distinct third roots of unity. 
\end{example}

\subsection*{The Algebraic Heisenberg Group}\quad\\
It is a classical result in the theory of elliptic curves that translation by a $3$--torsion point on a smooth cubic 
is afforded by a projective linear transformation; see \cite[\S 5 Case (b)]{Mum66}. 
We first recall the precise result and then show that the action of the relevant algebraic Heisenberg group 
lifts to a free action on the Moore matrices.

\begin{defn}
\label{Heis}
Let $K$ be a field that contains three distinct third roots of unity, 
$\mu_{3}(K) =\{1,\omega,\omega^{2}\}\leqslant K^{*}$ with $\omega^{3}=1$. In terms of the matrices
\begin{align*}
\mathsf\Sigma &= 
\left(
\begin{array}{ccc}
0&0&1\\
1&0&0\\
0&1&0
\end{array}
\right)\,,\quad
T=
\left(
\begin{array}{ccc}
1&0&0\\
0&\omega&0\\
0&0&\omega^{2}
\end{array}
\right)\in \SL(3,K)\,,
\end{align*}
of order $3$, the {\em algebraic Heisenberg group\/} is 
\begin{align*}
\mathsf {Heis}_{3}(K) &= \{\mu T^{i}\mathsf\Sigma^{j}\mid \mu\in K^{*}; i,j\in\ZZ/3\ZZ\}\leqslant \GL(3,K)\,.
\end{align*}
That this is indeed a subgroup is due to the equality $\mathsf \Sigma T = \omega T\mathsf \Sigma$.
This same equality also shows that $\mathsf {H}_{3}$, the subset of $\mathsf {Heis}_{3}$, where $\mu$
is restricted to powers of $\omega$, is a finite subgroup of $\SL(3,K)$ of order $27$.
\end{defn}

The crucial property of the algebraic {Heis}eisenberg group is then the following.

\begin{proposition}
Let $a=(a_{0}, a_{1}, a_{2})$ as before represent a point $\bfa\in E$ with $a_{0}a_{1}a_{2}\neq 0$.
For $a'=(a'_{0}, a'_{1}, a'_{2})\in K^{3}$ the following are equivalent.
\begin{enumerate}[\rm (1)]
\item 
$a'$ represents a point $\bfa'\in E$ with $3\cdot_{E}\bfa' = 3\cdot_{E}\bfa$.
\item
$a'$ is in the $\mathsf {Heis}_{3}$--orbit of $a$, thus, $a'\in \mathsf {Heis}_{3}\cdot a$.
\item
The Moore matrices $M_{a,x}$ and $M_{a',x}$ are equivalent.
\end{enumerate}
Moreover the action of $\mathsf {Heis}_{3}$ on the Moore matrices is free.
\end{proposition}

\begin{proof}
The equivalence of (1) and (2) is, of course, classical. For $(2)\Longrightarrow (3)$, it suffices
to show that the Moore matrices for $a,  T(a), \mathsf\Sigma(a)$ and $\mu a, \mu\in K^{*}$, are equivalent. 
This is obvious for $\mu a$ as $M_{\mu a,x}=(\mu\id_{3}) M_{a,x}$.
For $T(a) = (a_{0}, \omega a_{1}, \omega^{2} a_{2})$, the Moore matrix is
\begin{align*}
M_{T(a),x}&=
(\omega^{i+j}a_{i+j}x_{i-j})_{i,j\in\ZZ/3\ZZ}
=
(\omega^{i}a_{i+j}x_{i-j}\omega^{j})_{i,j\in\ZZ/3\ZZ}
= T\cdot M_{a,x}\cdot T\,.
\end{align*}
Thus, $M_{T(a),x}$ is equivalent to $M_{a,x}$.
For $\mathsf\Sigma(a) =(a_{2},a_{0},a_{1})$ one verifies
\begin{align*}
M_{(a_{2},a_{0}, a_{1}),x} &= (a_{i+j-1}x_{i-j})_{i,j\in\ZZ/3\ZZ} = (a_{(i+1)+(j+1)}x_{(i+1)-(j+1)})_{i,j\in \ZZ/3\ZZ}
=\mathsf\Sigma^{-1}\cdot M_{a,x}\cdot \mathsf\Sigma\,,
\end{align*}
whence $M_{\mathsf\Sigma(a),x}$ is indeed as well equivalent to $M_{a,x}$.

It remains to prove $(3)\Longrightarrow (1)$. If $M_{a,x}$ is equivalent to $M_{a',x}$ then these two matrices 
have the same determinant up to a nonzero scalar. This shows that $a'$ represents a point on $E$ along with $a$
and that $a'_{0}a'_{1}a'_{2}\neq 0$. 

Now write
\begin{align*}
M_{a,x} &= M_{0}x_{0} + M_{1}x_{1} + M_{2}x_{2}\,,
\end{align*}
where $M_{i} = \partial M_{a,x}/\partial x_{i}\in \Mat_{3\times 3}(K)$.
The matrix $M_{0}= \diag(a_{0}, a_{2}, a_{1})$ is invertible by assumption and we set
$N_{i} = M_{0}^{-1}M_{i}\in \Mat_{3\times 3}(K)$ and $N = \sum_{i=0}^{2}N_{i}x_{i}= M_{0}^{-1}M_{a,x} $. 

For $a'$ with $\bfa'\in E$ and $a'_{0}a'_{1}a'_{2}\neq 0$, define $N_{i}', N'$ analogously 
with $a'$ replacing $a$. Then the matrices $M_{a,x} , M_{a',x}$ are equivalent 
under the action of $\sfG_{3}$ if, and only if $N$ and $N'$ are equivalent under that action.
As $N_{0}=N_{0}'={\bf 1}_{3}$, the identity matrix, the matrices $N,N'$ are equivalent with respect to $\sfG_{3}$
if, and only if, $N$ and $N'$ are equivalent under conjugation by a matrix $P\in\GL(3, K)$, that is,
$N'=PNP^{-1}$, equivalently, $N_{1}'=PN_{1}P^{-1}$ and $N_{2}' = PN_{2}P^{-1}$.

In other words, the pairs of $3\times 3$ matrices $(N_{1}, N_{2})$ and $(N_{1}',N_{2}')$ are 
related by simultaneous conjugation. Clearly the trace functions $\tr(A_{1}\cdots A_{n})$, for any $n$--tuple
$A_{i}\in \{U,V\}, i=1,..., n$, are constant on the class of a pair $(U,V)\in\Mat_{3\times 3}(K)^{2}$ 
under simultaneous conjugation. Moreover, Teranishi \cite{Teranishi} showed that $11$ of these traces 
suffice to generate the ring of invariants. See \cite{Formanek} for a survey of these results, especially the 
list of the generating traces on the bottom of page 25.

We will not need any details of that invariant theory, but we easily extract from those classical results 
the traces that are relevant here. 

\begin{lemma}
\label{equiv}
With notation as just introduced, set further $a_{ij}=a_{i}/a_{j}$ for $i,j=0,1,2$.
\begin{enumerate}[\rm (i)]
\item The matrices $N_{1}, N_{2}$ have the form
\begin{align*}
N_{1}= 
\left(
\begin{array}{ccc}
0&0&a_{20}\\
a_{12}&0&0\\
0&a_{01}&0
\end{array}
\right)\,,\quad
N_{2}= 
\left(
\begin{array}{ccc}
0&a_{10}&0\\
0&0&a_{02}\\
a_{21}&0&0
\end{array}
\right)\,.
\end{align*}
\item Taking traces yields
\begin{align*}
\tr ((N_{1}N_{2})^{2})&
=\frac{a_{0}^{6}+a_{1}^{6}+a_{2}^{6}}{a_{0}^{2}a_{1}^{2}a_{2}^{2}}\\
\tr (N_{1}^{2}N_{2}^{2})&
=\frac{a_{0}^{3}a_{1}^{3}+a_{0}^{3}a_{2}^{3}+a_{1}^{3}a_{2}^{3}}{a_{0}^{2}a_{1}^{2}a_{2}^{2}}
\\
\tr (N_{1}N_{2}N_{1}^{2}N_{2}^{2})&
=\frac{a_{0}^{6}a_{1}^{3}+a_{1}^{6}a_{2}^{3}+a_{2}^{6}a_{0}^{3}}{a_{0}^{3}a_{1}^{3}a_{2}^{3}}\,.
\end{align*}
\end{enumerate}
\end{lemma}

\begin{proof}
Straightforward verification.
\end{proof}
Combining item (ii) in this Lemma with Corollary \ref{3fold} shows that equivalence of $M_{a,x}$ and $M_{a',x}$ 
forces $3\cdot_{E}\bfa'=3\cdot_{E}\bfa$.

As for the final claim, this follows from Beauville's result that the action of $\sfG_{3}$ on linear matrices is free.
\end{proof}

\begin{cor}
\label{alg{Heis}eis}
The subgroup of $\sfG_{3}$ that transforms Moore matrices into such is isomorphic to the
algebraic Heisenberg group $\mathsf {Heis}_{3}$.\qed
\end{cor}

\begin{remark}
In light of the preceding result, we sometimes write $M_{\bfa, x}$ to denote any representative of the equivalence 
class of $M_{a,x}$ under the action of $\mathsf {Heis}_{3}$, with $\bfa\in E$ as before representing the point
underlying $a\in K^{3}$.
\end{remark}

It is indeed the representation theory of the Heisenberg groups that allows us to finish the proof 
of Theorem A. Instead of working with the algebraic Heisenberg groups, it suffices to restrict to the
finite Heisenberg groups and their representations.

\subsection*{The Schr\"odinger Representations of the Finite Heisenberg Groups}\quad\\
\begin{sit}
The general Heisenberg group $\mathsf H(R)$ over a commutative ring $R$ is usually understood to be
the subgroup of unipotent upper triangular $3\times 3$ matrices in $\GL(3,R)$. 
For $R=\ZZ/n\ZZ$, $n\geqslant 1$ an integer, we call these the {\em finite Heisenberg groups\/} and abbreviate 
$\mathsf H_{n}= \mathsf H(\ZZ/n\ZZ)$. The group $\mathsf H_{n}$ is of order $n^{3}$ and admits the presentation
\begin{align*}
\mathsf H_{n} =\langle \sigma,\tau\mid [\sigma,[\sigma,\tau]] =
[\tau,[\sigma,\tau]]= \sigma^{n}=\tau^{n}=1\rangle\,.
\end{align*}
%Note: in many sources, the relation $[\sigma,\tau]^{n} =1$ is added, but this is redundant:
%\begin{align*}
%\tau&= \sigma^{n}\tau = \sigma^{n-1}\sigma\tau =\sigma^{n-1} [\sigma,\tau]\tau\sigma&&\text{obviously,}\\
%& = [\sigma,\tau]\sigma^{n-1}\tau\sigma &&\text{because $[\sigma,\tau]$ is central,}\\
%& = [\sigma,\tau]^{n}\tau\sigma^{n} &&\text{by induction,}\\
%&= [\sigma,\tau]^{n}\tau
%\end{align*}
%Cancelling $\tau$ from both sides yields $[\sigma,\tau]^{n} =1$.

Each element of $\mathsf H_{n}$ has a unique representation as $[\sigma,\tau]^{r}\sigma^{s}\tau^{t}$ with
$r,s,t\in\ZZ/n\ZZ$.

Note that $\mathsf H_{3}$ as defined here is indeed isomorphic to the group $\mathsf H_{3}$
that we exhibited as a subgroup of $\mathsf{Heis}_{3}$ above.
\end{sit}

\begin{sit}
Over a field $K$ that contains a primitive $n^{th}$ root of unity $\zeta\in K^{*}$, the group $\mathsf H_{n}$ 
carries the $K$--linear {\em Schr\"odinger representations} $\rho_{j}:\mathsf H_{n}\to \GL(n,K)$, parametrized by 
$j\in\ZZ/n\ZZ$, that in a suitable {\em Schr\"odinger basis\/}  $v_{i}, i\in \ZZ/n\ZZ,$ of a vector space $V$ of 
dimension $n$ over $K$ are given by
\begin{gather*}
\rho_{j}(\sigma)(v_{i})= v_{i-1}\,,\quad \rho_{j}(\tau)(v_{i})=\zeta^{ij}v_{i}\,,\quad i\in \ZZ/n\ZZ\,,
\intertext{and thus, for a general element,}
\rho_{j}([\sigma,\tau]^{r}\sigma^{s}\tau^{t})(v_{i})=\zeta^{j(it+r)}v_{i-s}\,.
\end{gather*}

In particular, the character $\chi_{j}$ of the representation $\rho_{j}$ satisfies
\begin{align*}
\chi_{j}([\sigma,\tau]^{r}\sigma^{s}\tau^{t})&=
\begin{cases}
0&\text{if $s\not\equiv 0\bmod n$ or $jt\not\equiv 0\bmod n$}\\
n\zeta^{jr}&\text{if $s\equiv jt\equiv 0\bmod n$.}
\end{cases}
\end{align*}
\end{sit}

\begin{remark}
For a complete and detailed discussion of the irreducible representations of the finite Heisenberg groups 
see \cite{GH01}.
\end{remark}

\begin{sit}
If $d\geqslant 2$ is a divisor of $n$, say $n=dm$, then the subgroup of $\mathsf H_{n}$ generated by
$\sigma^{m},\tau^{m}$ is a homomorphic image of $\mathsf H_{d}$, in that surely
$\sigma^{m}$ and $\tau^{m}$ are of order $d$, and these elements commute with
$[\sigma^{m},\tau^{m}] =  [\sigma,\tau]^{m^{2}}$. If we restrict the Schr\"odinger representation 
$\rho_{j}$ of $\mathsf H_{n}$ along the resulting homomorphism $\mathsf H_{d}\to \mathsf H_{n}$, then it 
decomposes in that the actions of $\sigma^{m}$ and $\tau^{m}$, given by
\begin{align*}
\rho_{j}(\sigma^{m})(v_{i}) &= v_{i-m}\,,\quad \rho_{j}(\tau^{m})(v_{i}) =(\zeta^{m})^{ij}v_{i}\,\quad\text{for 
$i\in\ZZ/n\ZZ$,} 
\end{align*}
yield the $\mathsf H_{d}$--subrepresentations 
\begin{align*}
W_{jk}&= \bigoplus_{i\equiv k\bmod m}Kv_{i}\subseteq V\quad\text{for $k\in \ZZ/m\ZZ$.}
\end{align*}
For $i=\alpha m +k$, in the basis
$w_{\alpha} = v_{\alpha m +k}$, for $\alpha=0,...,d-1,$ of $W_{jk}$ the action is given by
\begin{align*}
\rho_{j}(\sigma^{m})(w_{\alpha})&=
w_{\alpha-1}\,,\\
\rho_{j}(\tau^{m})(w_{\alpha})&= (\zeta^{m})^{j(\alpha m+k)}w_{\alpha}\,,\quad \text{and, for a general element}\\
\rho_{j}([\sigma^{m},\tau^{m}]^{r}(\sigma^{m})^{s}(\tau^{m})^{t})(w_{\alpha})&= 
\rho_{j}([\sigma,\tau]^{m^{2}r}\sigma^{ms}\tau^{mt})(w_{\alpha}) \\
&= (\zeta^{m})^{j((\alpha m+k)t+mr)}w_{\alpha-s}\,.
\end{align*}
The corresponding character is thus
\begin{align*}
\rho_{j}([\sigma^{m},\tau^{m}]^{r}(\sigma^{m})^{s}(\tau^{m})^{t}) 
&=
\begin{cases}
0&\text{if $s\not\equiv 0\bmod d$ or $jmt\not\equiv 0\bmod d$}\\
d(\zeta^{m})^{j(kt+mr)}&\text{if $s\equiv jmt\equiv 0\bmod d$.}
\end{cases}
\end{align*}
If $\gcd(d,m)=1$, then $jmt\equiv 0\bmod d$ if, and only if $jt\equiv 0\bmod d$ and one recognizes
the Schr\"odinger representation $\rho_{jm}$ of $\mathsf H_{d}$. Therefore, we have the following result.
\end{sit}

\begin{lemma}
\label{restrict}
For $d$ a positive divisor of $n$ with $\gcd(d,n/d)=1$, under the group homomorphism 
$\mathsf H_{d}\to \mathsf H_{n}$ described above the Schr\"odinger representation $\rho_{j}$ of 
$\mathsf H_{n}$ restricts to the direct sum of $n/d$ copies of the Schr\"odinger representation 
$\rho_{(jn/d)\bmod d}$ of $\mathsf H_{d}$.\qed
\end{lemma}

\begin{sit}
\label{action}
Returning to elliptic curves, let, more generally, $L$ be an ample line bundle on an abelian variety defined 
over an algebraically closed field $K$ whose characteristic does not divide the degree $n>0$ of $L$.
It is a deep result from the theory of abelian varieties; see \cite[Prop. 3.6]{Mum91} for the general case or
\cite{Hulek} for an explicit treatment over the complex numbers; 
that then the vector space of sections of $L$ comes naturally  equipped with the Schr\"odinger representation 
$\rho_{1}$ of $\mathsf H_{n}$ --- in fact this is the restriction of the Schr\"odinger representation of the 
larger algebraic Heisenberg group $\mathsf{Heis}_{n}$ that is defined in analogous fashion to $\mathsf{Heis}_{3}$. 

In case $L$ is a line bundle on an elliptic curve $E$ over $K$  this representation lifts the translation by 
$n$--torsion points on $E$, thus, the action of $E[n]\cong \ZZ/n\ZZ\times \ZZ/n\ZZ$ on $\PP(H^{0}(E,L))$ 
to an action by linear automorphisms on $V=H^{0}(E,L)$. For an elliptic curve $E$, embedded as
a smooth projective plane cubic curve, the special case $L=\calo_{E}(1)$ with $n=\deg L = 3$ was discussed 
in detail above.
\end{sit}

Using the preceding Lemma, the following result is an easy consequence of the fundamental fact
just recalled. 

\begin{proposition}
\label{tensor}
Let $\call,\call'$ be locally free sheaves of degree $3$ and $\call''$ a locally free sheaf of degree $6$ 
on an elliptic curve $E$ over an algebraically closed field $K$ whose characteristic does not divide $6$.
\begin{enumerate}[\rm (a)]
\item 
\label{tensora}
Restricting the translations by $6$--torsion points to the $3$--torsion points restricts the
representation $\rho_{1}$ of $\mathsf H_{6}$ on $H^{0}(E,\call'')$ to the direct sum of two copies of the
Schr\"odinger representation $\rho_{2}$ of $\mathsf H_{3}$.

\item
\label{tensorb}
The tensor product $H^{0}(E,\call)\otimes_{K}H^{0}(E,\call')$ of the Schr\"odinger representations
$\rho_{1}$ of $\mathsf H_{3}$ on each of the two factors decomposes as the direct sum of three copies of
the Schr\"odinger representation $\rho_{2}$ of $\mathsf H_{3}$.

\item 
\label{tensorc}
With $\call''=\call\otimes_{\calo_{E}}\call'$, the natural multiplication map on global sections
%$H^{0}(E,\call)\otimes_{K}H^{0}(E,\call') \lto H^{0}(E,\call'')$ 
represents a surjective $\mathsf H_{3}$--equivariant homomorphism
\begin{align*}
H^{0}(E,\call)\otimes_{K}H^{0}(E,\call') \lto {\res}{\big\downarrow^{\mathsf H_{6}}_{\mathsf H_{3}}} H^{0}(E,\call'')\,.
\end{align*}
In particular, the kernel of that homomorphism is a Schr\"odinger representation $\rho_{2}$ of
$\mathsf H_{3}$.
\end{enumerate}
\end{proposition}

\begin{proof}
Part (\ref{tensora}) is Lemma \ref{restrict} applied to the case $j=1, n=6, d=3$, thus $n/d=2$.

For Part (\ref{tensorb}), let $(x_{0}, x_{1}, x_{2})$ be a Schr\"odinger basis of $V=H^{0}(E,\call)$ and
$(y_{0}, y_{1}, y_{2})$ be a Schr\"odinger basis of $V'=H^{0}(E,\call')$. With $x_{i}y_{j}=x_{i}\otimes y_{j}$
and $a_{0},a_{1},a_{2}\in K$, the tensor $f_{0} = a_{0}x_{0}y_{0}+a_{1}x_{2}y_{1}+a_{2}x_{1}y_{2}$
is a fixed vector for the action of $\tau' =\rho_{1}(\tau)\otimes\rho_{1}(\tau)$ on $V\otimes_{K}V'$.
Abbreviating also $\sigma' =\rho_{1}(\sigma)\otimes\rho_{1}(\sigma)$, with $f_{-i}=(\sigma')^{i}(f_{0}),
i\in\ZZ/3\ZZ$, one has
\begin{align*}
f_{0} &= a_{0}x_{0}y_{0}+a_{1}x_{2}y_{1}+a_{2}x_{1}y_{2}\,,&\tau'(f_{0})&= f_{0}\\
f_{1} &= a_{2}x_{2}y_{0}+a_{0}x_{1}y_{1}+a_{1}x_{0}y_{2}\,,& \tau'(f_{1})&=\omega^{2} f_{1}\,,\\
f_{2} &= a_{1}x_{1}y_{0}+a_{2}x_{0}y_{1}+a_{0}x_{2}y_{2}\,,\quad& \tau'(f_{2}) &= \omega f_{2}\,.
\end{align*}
Therefore, $f_{0}, f_{1}, f_{2}$ form indeed a Schr\"odinger basis for a representation of 
$\mathsf H_{3}$ that is equivalent to $\rho_{2}$ as soon as $(a_{0},a_{1},a_{2})\neq (0,0,0)\in K^{3}$. 
Choosing in turn $(a_{0},a_{1},a_{2}) = e_{i}$, for $i\in\ZZ/3\ZZ$ and $(e_{i})_{i=0,1,2}$ the standard 
basis of $K^{3}$, it follows that indeed 
$\rho_{1}\otimes_{K}\rho_{1}\cong \rho_{2}^{\oplus 3}$ as $\mathsf H_{3}$--representations --- which fact 
could have been established as well by just looking at the corresponding group characters. 
The reader will note that viewed as trilinear forms, the $f_{i}$ are precisely the forms from (\ref{ATV}) above.

In Part (\ref{tensorc}), surjectivity of the multiplication map is well known and the 
$\mathsf H_{3}$--equivariance follows as translation is compatible with tensor products,
$t_{\bfa}(\call)\otimes_{\calo_{E}}t_{\bfa}(\call')\cong t_{\bfa}(\call\otimes_{\calo_{E}}\call')$ for any
point $\bfa\in E$. Applied to $3$--torsion or $6$--torsion points and using that translations by those points manifest 
themselves through the Schr\"odinger representation $\rho_{1}$ of $\mathsf H_{3}$, respectively $\mathsf H_{6}$,
the proof of the Proposition is complete.
\end{proof}

\subsection*{End of the Proof of Theorem A}
\begin{sit}
Let $L$ be a line bundle of degree $0$ on the smooth cubic curve 
$E\subset \PP^{2}$ with defining equation $f=0$.
According to Beauville's result stated above in Theorem \ref{beau}, if $H^{0}(E,L)=0$ then  
there exists a $3 \times 3$ linear matrix $M$ such that $f=\det M$ and an exact
sequence of $\calo_{\PP^{2}}$--modules
\begin{align*}
\xymatrix{
0\ar[r]&\calo_{\PP^{2}}(-2)^{3}\ar[r]^-{M}&\calo_{\PP^{2}}(-1)^{3}\ar[r]&L\ar[r]&0\,.
}
\end{align*}
Twisting this sequence by $\calo_{E}(1)$, respectively $\calo_{E}(2)$, and taking sections, one can identify
this exact sequence as
\begin{align*}
\xymatrix{
0\ar[r]&\calo_{\PP^{2}}(-2)\otimes_{K}W\ar[r]^-{M}&\calo_{\PP^{2}}(-1)\otimes_{K}H^{0}(E,L(1))\ar[r]&L\ar[r]&0\,,
}
\end{align*}
where $W$ is the kernel of the $\mathsf H_{3}$--equivariant multiplication map 
\[
H^{0}(E,\calo_{E}(1))\otimes_{K}H^{0}(E,L(1))\lto \res\big\downarrow _{\mathsf H_{3}}^{\mathsf H_{6}}H^{0}(E,L(2))
\] 
as in Proposition \ref{tensor}(\ref{tensorc}) above for $\call=\calo_{E}(1), \call'=L(1)$. 

Choosing Schr\"odinger bases $f_{0}, f_{1}, f_{2}$ for $W$,
$x_{0}, x_{1},x_{2}$ for $H^{0}(E,\calo_{E}(1))$, and $y_{0}, y_{1}, y_{2}$ for $H^{0}(E,L(1))$ as in the proof of
Proposition \ref{tensor}(\ref{tensorb}), $M$ becomes identified with a Moore matrix 
$M_{a,x}=(a_{i+j}x_{i-j})_{i,j\in\ZZ/3\ZZ}$ for some $a= (a_{0}, a_{1}, a_{2})\in K^{3}$.
As $\det M\doteq f$ by Beauville's result, it follows that $\bfa\in E$ with $a_{0}a_{1}a_{2}\neq 0$.
This completes the proof of Theorem A from the introduction.\qed
\end{sit}

\section{Proof of Theorem B}
The starting point is the following result from Atiyah's seminal paper \cite{Atiyah57}.
\subsection*{A Result of Atiyah and Ulrich Bundles}
\begin{theorem}[cf. Atiyah {\cite[Thm. 5 (ii)]{Atiyah57}}]
\label{Atiyah}
Let $F$ be an indecomposable vector bundle of rank $2$ on an elliptic curve $E$ over a field $K$. If
$\deg F=0$, then there exists a unique line bundle $L$ of degree $0$ that fits into an exact sequence
of vector bundles
\begin{align*}
\xymatrix{
0\ar[r]&L\ar[r]&F\ar[r]&L\ar[r]&0
}
\end{align*}
Moreover, $H^{0}(E,F)=0$ if, and only if, $H^{0}(E,L)=0$.

Conversely, if $L$ is a line bundle of degree $0$ then there exists an indecomposable
vector bundle $F$, unique up to isomorphism and necessarily of rank $2$ and degree $0$, that fits 
into such an exact sequence.\qed
\end{theorem}

\begin{sit}
By our Theorem A we know that a line bundle of degree $0$ with no nonzero sections is obtained as
$L=\coker M_{a,x}$, where $M_{a,x}$ is a Moore matrix, $a\in K^{3}$ representing
a point $\bfa\in E$ on the elliptic curve $E\subset \PP^{2}$ with $a_{0}a_{1}a_{2}\neq 0$. 
We fix these data in the following.
\end{sit}

Let $S=K[x_{0},x_{1},x_{2}]$ be the homogeneous coordinate ring of $\PP^{2}(K)$, with its homogeneous 
components $S_{m}=H^{0}(\PP^{2}, \calo_{\PP^{2}}(m))$ the vector space over $K$ of homogeneous polynomials of degree $m\in\ZZ$.

Applying the functor $\Gamma_{*} = \oplus_{i\in\ZZ}H^{0}(\PP^{2},(\ )(i))$ to the exact sequence of 
coherent $\calo_{\PP^{2}}$--modules
\begin{align*}
\xymatrix{
0\ar[r]&\calo_{\PP^{2}}(-2)^{3}\ar[r]^-{M_{a,x}}&\calo_{\PP^{2}}(-1)^{3}\ar[r]&L\ar[r]&0\,,
}
\end{align*}
yields a short exact sequence of graded $S$--modules
\begin{align*}
\xymatrix{
0\ar[r]&S(-2)^{3}\ar[r]^-{M_{a,x}}&S(-1)^{3}\ar[r]&\Gamma_{*}(L)\ar[r]&0\,.
}
\end{align*}

The module $\mathbf L= \Gamma_{*}(L)$, cokernel of the map between graded free $S$--modules 
represented by $M_{a,x}$, is an {\em Ulrich module\/} of rank one over the homogeneous coordinate ring 
$R=S/(f)$ of $E$, and each Ulrich module over $R$ of rank one (and generated in degree $1$)
can be so realized by Theorem A. 

\subsection*{Matrix Factorizations and Extensions}\quad\\
In view of Atiyah's result cited above, our aim here is to find a similar description for Ulrich modules over $R$ 
of rank two, namely the one stated in Theorem B. To simplify notation a bit, we fix for now the point $a$ and 
set $A=M_{a,x}, B=M_{a,x}^{\adj}$, viewed as matrices over $S$. The pair $(A,B)$ represents a matrix 
factorization of $f=\det A\in S$ and so, by \cite{Ei}, $\mathbf L$ admits the graded $R$--free resolution
\begin{align*}
\xymatrix{
0&\ar[l] \mathbf L&\ar[l]R(-1)^{3}&\ar[l]_-{A}R(-2)^{3}&\ar[l]_-{B}R(-4)^{3}&\ar[l]_-{A}R(-5)^{3}&\ar[l]_-{B}\cdots
}
\end{align*}
that is $2$--periodic up to the shift in degrees by $-\deg f=-3$.

\begin{sit}
Now consider an element%
\footnote{We write $\Extgr_{R}$ for the extensions in the category of graded $R$--modules with degree--preserving 
$R$--linear maps.}
 of $\Extgr^{1}_{R}(\mathbf L, \mathbf L(m))$ for some $m\in\ZZ$. It can be represented
by a homotopy class of morphisms between graded free resolutions and, invoking again 
\cite{Ei}, such morphisms and their
homotopies can again be chosen to be $2$--periodic so that one has a diagram as follows 
\begin{align*}
\xymatrix{
\mathbf L\ar[d]&0&
\ar[l]R(-1)^{3}\ar[d]\ar@{.>}[dr]|-{U}&
\ar[l]_-{A}R(-2)^{3}\ar[d]^-{C}\ar@{.>}[dr]|-{V}&
\ar[l]_-{B}R(-4)^{3}\ar[d]^-{D}\ar@{.>}[dr]|-{U}&
\ar[l]_-{A}R(-5)^{3}\ar[d]^-{C}\ar@{.>}[dr]|-{V}&
\ar[l]_-{B}\cdots\\
\mathbf L(m)[1]&0&\ar[l]0&\ar[l]R(m-1)^{3}&\ar[l]^-{-A}R(m-2)^{3}&\ar[l]^-{-B}R(m-4)^{3}&
\ar[l]^-{-A}\cdots
}
\end{align*}
Here
\begin{itemize}
\item $C\in\Mat_{3\times 3}(S_{m+1}),D\in\Mat_{3\times 3}(S_{m+2})$ are $3\times 3$ matrices whose entries are 
homogeneous poynomials of the indicated degrees.
\item The pair of matrices $(C,D)$ satisfies
$AD+CB = \mathbf 0 = DA+BC$
over $S$, with $ \mathbf 0$ the zero matrix, and so defines a morphism of complexes over $R$. 
\item
$U,V\in\Mat_{3\times 3}(S_{m})$ represent the possible homotopies, in that the morphisms of complexes
$\mathbf L\to \mathbf L(m)$ induced by
\begin{align*}
C'&= C+ UA-AV\,,\\
D'&= D + VB-BV
\end{align*}
run through the homotopy class of $(C,D)$ for the various choices of $U,V$.
\end{itemize}
\end{sit}

\begin{sit}
Given a pair of matrices $(C,D)$ with $AD+CB=\mathbf 0=DA+BC$ as above, the block matrices
\begin{align*}
\left(
\begin{array}{cc}
A&C\\
0&A
\end{array}
\right)\,,\quad
\left(
\begin{array}{cc}
B&D\\
0&B
\end{array}
\right)
\end{align*}
constitute a matrix factorization of $f$ and give rise to the commutative diagram of graded $S$--modules
with exact rows and columns
\begin{align*}
\xymatrix{
&0\ar[d]&&0\ar[d]&0\ar[d]\\
0\ar[r]&S(m-2)^{3}\ar[d]_{in_{1}}\ar[rr]^-{A}&&S(m-1)^{3}\ar[r]\ar[d]_{in_{1}}&\mathbf L(m)\ar[r]\ar[d]&0\\
0\ar[r]&S(m-2)^{3}\oplus S(-2)^{3}\ar[d]_{pr_{2}}
\ar[rr]^-{
{\left(
\begin{array}{cc}
\scriptstyle A&\scriptstyle C\\
\scriptstyle 0&\scriptstyle  A
\end{array}
\right)}}
&&S(m-1)^{3}\oplus S(-1)^{3}\ar[r]\ar[d]_{pr_{2}}
&\mathbf F\ar[r]\ar[d]&0\\
0\ar[r]&S(-2)^{3}\ar[d]\ar[rr]^-{A}&&S(-1)^{3}\ar[r]\ar[d]&\mathbf L\ar[r]\ar[d]&0\\
&0&&0&0
}
\end{align*}
with the rightmost column representing the extension defined by $(C,D)$ over $R$.
\end{sit}

The following observation cuts down considerably on the work of finding solutions to
the equations $AD+CB=\mathbf 0 = DA+BC$, whenever $(A,B)$ is a matrix factorization of a non-zero-divisor $f$ in a commutative ring $S$.

\begin{lemma}
Assume $A,B\in \Mat_{n\times n}(S)$ is a matrix factorization of a non-zero-divisor $f\in S$, in that
$AB=f\id_{n}=BA$. For a matrix $C\in  \Mat_{n\times n}(S)$ the following are equivalent.
\begin{enumerate}[\rm(a)]
\item 
\label{D}
There exists a matrix $D\in  \Mat_{n\times n}(S)$ such that $AD+CB=\mathbf 0$.
\item
\label{D'}
There exists a matrix $D'\in  \Mat_{n\times n}(S)$ such that $D'A+BC=\mathbf 0$.
\item
\label{D''}
There exists a matrix $D''\in  \Mat_{n\times n}(S)$ such that $fD''+BCB=\mathbf 0$. Equivalently, each entry of 
$BCB\in \Mat_{n\times n}(S)$ is divisible by $f$.
\end{enumerate}
If either equivalent condition holds then $D=D'=D''$ and that matrix is the unique one satisfying $fD=-BCB$.
Moreover, one can recover $C$ from $D$ in that $C$ is the unique matrix such that $fC = -ADA$.
\end{lemma}

\begin{proof}
If $AD+CB=\mathbf 0$ then multiplying from the left with $B$ yields
\begin{align*}
0 = BAD + BCB = fD + BCB\,,
\end{align*}
whence $BCB\equiv \mathbf 0\bmod (f)$ and one can take $D''=D$. 
Conversely, if that congruence holds then there exists a matrix $D''$ 
with $fD''=-BCB$. Multiplying this equation with $A$ from the left results in
\begin{align*}
AfD'' + ABCB = f(AD''+CB) =\mathbf 0\,.
\end{align*}
As $f$ is a non-zero-divisor, this implies $AD''+CB=\mathbf 0$, whence one can take $D=D''$.
Thus (\ref{D}) $\Longleftrightarrow$ (\ref{D''}). The equivalence  (\ref{D'}) $\Longleftrightarrow$ (\ref{D''}) is
completely analogous.
Uniqueness of $D$ follows as $AD_{1} + CB= AD_{2} + CB$ implies $A(D_{1}-D_{2})=0$. 
Now the linear map represented by $A$ is injective because $f\id_{n} = BA$ and multiplication with 
$f$ is injective by assumption. 
Thus, $D_{1}=D_{2}$ as claimed and, in particular, $D=D''$ must hold. Analogously one must have $D'=D''$. 

Concerning the final assertion, multiply the equation $fD=-BCB$ on both sides with $A$ to obtain
$ADAf = -fCf$, thus, $fC=-ADA$, as $f$ is a non-zero-divisor. Uniqueness of $C$ follows as above.
\end{proof}
If $A$, as in our case of interest, is a determinantal representation of a reduced polynomial,
one can reduce the description of extensions further.
\begin{lemma}
If the determinant $f=\det A\in S$, for a matrix $A \in \Mat_{n\times n}(S)$, is reduced, then with $B=A^{\adj}$ 
one has for any matrix $C\in \Mat_{n\times n}(S)$ that
\begin{align*}
BCB \equiv \tr(BC)B\bmod (f)\,,
\end{align*}
and $BCB\equiv \mathbf 0\bmod (f)$ if, and only if, $\tr(BC)\equiv 0\bmod (f)$.
\end{lemma}

\begin{proof}
As $f$ is reduced, it is generically regular. For a regular point $x\in V(f)$ this implies that
$\rank A(x) = n-1$, thus, $\rank B(x) = 1$, as the cokernel of $A$ is locally free of rank $1$ at such point.
Accordingly there are vectors $u,v\in k(x)^{n}$ such that $B(x) = u^{T}\cdot v$. Therefore,
\begin{align*}
B(x)C(x)B(x) =  u^{T}\cdot v\cdot C(x)\cdot u^{T}\cdot v\,.
\end{align*}
Now $v\cdot C(x)\cdot u^{T}$ is an element of the residue field $k(x)$ at $x$ and so, 
considering it as a $1\times 1$ matrix,
\begin{align*}
v\cdot C(x)\cdot u^{T} =\tr(v\cdot C(x)\cdot u^{T})=\tr(u^{T}\cdot v\cdot C(x)) =\tr(B(x)C(x))\,.
\end{align*}
Embedding this observation into the right-hand side of the previous equality, it follows that
\begin{align*}
B(x)C(x)B(x) =  u^{T}\cdot v\cdot C(x)\cdot u^{T}\cdot v = \tr(B(x)C(x)) u^{T}\cdot v = \tr(B(x)C(x)) B(x)\,.
\end{align*}
Therefore, $BCB - \tr(BC)B$ vanishes at each regular point of $\{f=0\}$, thus, it vanishes everywhere
on that hypersurface. Moreover, $B(x)\neq \mathbf 0$ at each regular point, whence at such points
$BCB(x) = \mathbf 0$ if, and only if, $\tr(B(x)C(x)) = 0$. The claim follows.
\end{proof}

\begin{sit}
\label{ext1}
Putting these facts together, we get the following description of the extension groups we are interested in here:
\begin{align}
\tag{$*$}
\label{ext1*}
\Extgr^{1}_{R}(\mathbf L,\mathbf L(m)) \cong
\frac{\left\{C\in \Mat_{3\times 3}(S_{m+1})\mid \tr(M_{a,x}^{\adj}\cdot C)\equiv 0\bmod f\right\}}
{\left\{UM_{a,x}-M_{a,x}V\mid U,V\in  \Mat_{3\times 3}(S_{m})\right\}}
\end{align}

This description shows immediately that $\Extgr^{1}_{R}(\mathbf L,\mathbf L(m))=0$ when $m<-1$, because
there is then no nonzero choice for $C$. It also shows that there are no nonzero homotopies for $m = -1$, 
a fact we will exploit below. Concerning shifts by $m\geqslant -1$, we determine the size of the extension group
directly in terms of possible  Yoneda extensions, that is, short exact sequences, as follows. 
Given a short exact sequence
\begin{align*}
\xymatrix{
0\ar[r]&\mathbf L(m)\ar[r]&\mathbf M\ar[r]&\mathbf L\ar[r]&0\hphantom{\,.}
}
\end{align*}
of graded $R$--modules, sheafifying it yields a short exact sequence of $\calo_{E}$--modules,
\begin{align*}
\xymatrix{
0\ar[r]&L(m)\ar[r]& M\ar[r]& L\ar[r]&0\,.
}
\end{align*}
Conversely, applying $\Gamma_{*}$ to such a short exact sequence of
$\calo_{E}$--modules with $m\geqslant -1$ returns a short exact sequence of graded $R$--modules 
as above in that the connecting homomorphism in cohomology $H^{0}(E,L(i))\to H^{1}(E,L(m+i))$ is $0$ for 
each $i\in\ZZ$. Indeed, for $i\leqslant 0$, one has $H^{0}(E,L(i))=0$, while for $i>0$ one has 
$i+m\geqslant 0$, whence $H^{1}(E,L(m+i))=0$. 

It follows that $\Gamma_{*}$ and sheafification yield inverse isomorphisms between
$\Ext^{1}_{E}(L,L(m))$ and $\Extgr^{1}_{R}(\mathbf L,\mathbf L(m))$ for $m\geqslant -1$.

Further,
$\Ext^{1}_{E}(L,L(m))\cong H^{1}(E,\calo_{E}(m))$ vanishes for $m > 0$,
while for $m=-1$ that vector space is Serre--dual to $H^{0}(E,\calo_{E}(1))\cong S_{1}$ and for $m=0$, 
of course, $H^{1}(E,\calo_{E})\cong K$.
\end{sit}

We thus have the following result.

\begin{lemma}
\label{dimensions}
Let $\mathbf L$ be the Ulrich module that is the cokernel of the Moore matrix $M_{a,x}$
as in Theorem A.
The vector spaces of graded self-extensions of $\mathbf L$ over $R$ satisfy
\begin{align*}
\dim_{K}\Extgr^{1}_{R}(\mathbf L,\mathbf L(m)) &=
\begin{cases}
3&\text{for $m=-1$,}\\
1&\text{for $m= 0$,}\\
0&\text{else.}
\end{cases}
\end{align*}
\qed
\end{lemma}

With these preparations we now determine $\Extgr^{1}_{R}(\mathbf L,\mathbf L(-1))$.

\begin{proposition}
Let $\mathbf L$ be the cokernel of a Moore matrix $M_{a,x}:S(-2)^{3}\lto S(-1)^{3}$ for $a\in K^{3}$
 with $a_{0}a_{1}a_{2}\neq 0$
representing a point $\bfa\in E$. 
The three-dimensional vector space
$\Extgr^{1}_{R}(\mathbf L,\mathbf L(-1))$ over $K$ is isomorphic to the space of specialized Moore matrices
\begin{align*}
M_{(a_{0}(a_{2}^{3}-a_{1}^{3}), a_{2}(a_{1}^{3}-a_{0}^{3}), a_{1}(a_{0}^{3}-a_{2}^{3})),(s,t,u)}\,,
\quad (s,t,u)\in K^{3}\,.
\end{align*}
Note that $(a_{0}(a_{2}^{3}-a_{1}^{3}), a_{2}(a_{1}^{3}-a_{0}^{3}), a_{1}(a_{0}^{3}-a_{2}^{3}))$ represents
the point $2\cdot_{E}\bfa\in E$, whence this set of matrices consists of all $K$--rational specializations of
$M_{2\cdot_{E}\bfa,x}$.
\end{proposition}

\begin{proof}
Set $b=(a_{0}(a_{2}^{3}-a_{1}^{3}), a_{2}(a_{1}^{3}-a_{0}^{3}), a_{1}(a_{0}^{3}-a_{2}^{3}))$ and note that
\begin{align*}
M_{b,(s,t,u)} &= M_{b,(1,0,0)}s+M_{b,(0,1,0)}t+M_{b,(0,0,1)}u\,. 
\end{align*}
Because $b\neq (0,0,0)$, the matrices $M_{0}=M_{b,(1,0,0)}, M_{1}=M_{b,(0,1,0)}, M_{2}= M_{b,(0,0,1)}$ 
are clearly linearly independent in the vector space $\Mat_{3\times 3}(K)$ in that their nonzero entries 
are located at different positions in these matrices. Further, as $\tr(B\cdot -)$ is an $S$--linear function on 
$\Mat_{3\times 3}(S)$, and $\Extgr^{1}_{R}(\mathbf L,\mathbf L(-1))$ is known to be of dimension $3$ over $K$, 
it suffices to show that for each of the three matrices $M_{i}$ one has 
$\tr(BM_{i})\equiv 0\bmod(f)$, where $B= M_{a,x}^{\adj}$. In fact, as the entries of $BM_{i}$ are quadratic 
polynomials, but $f$ is of degree $3$, the congruence is equivalent to $\tr(BM_{i})=0$.

One now verifies this easily directly for the three matrices in question.

For a more conceptual explanation of the identities $\tr(BM_{i})=0$, note that, say,
$M_{0} = M_{b,(1,0,0)}=\diag(b)$ is a diagonal matrix with the coordinates $b$ on the diagonal
representing $2\cdot_{E}\bfa$. On the other hand, the diagonal elements in $B=M_{a,x}^{\adj}$ involve only the 
quadratic monomials $x_{0}^{2}$ and $x_{1}x_{2}$, and the coefficients of $x_{0}^{2}$ along the diagonal are the 
entries from the third row, those of $x_{1}x_{2}$ the entries from the first row of $M_{\iota(a),a}$; see
(\ref{double}). 
The column vector $b^{T}$ spans the kernel of that matrix by construction. 
As the trace $\tr(BM_{0})$ is the scalar product of the two diagonals, the vanishing of the trace becomes obvious. 
The case of the remaining two matrices yields to analogous arguments.
\end{proof}
\subsection*{The Selfextensions of an Ulrich Line Bundle}\quad\\
Next we turn to $\Extgr^{1}_{R}(\mathbf L,\mathbf L)$, the extension group we are really interested in.
\begin{defn}
With notation as in the preceding proof, set
\begin{align*}
M_{b,y} &= M_{b,(1,0,0)}y_{0}+M_{b,(0,1,0)}y_{1}+M_{b,(0,0,1)}y_{2}\,,
\end{align*}
for $y=(y_{0}, y_{1}, y_{2})\in S_{1}^{3}$, a vector of linear forms from $S$, and define the {\em divergence\/} 
of $M_{b,y}$ to be
\begin{align*}
\mathsf{div}(M_{b,y}) &=
\frac{\partial y_{0}}{\partial x_{0}} + \frac{\partial y_{1}}{\partial x_{1}} + \frac{\partial y_{2}}{\partial x_{2}}\in K\,.
\end{align*}
Note in particular that $\mathsf{div}(M_{b,x}) = 3\in K$, thus, is not zero when $\car(K)\neq 3$. For a 
characteristic--free statement, note that $\mathsf{div}(M_{b,\iota(x)}) = 1\in K$.
\end{defn}

\begin{theorem}
Let $\mathbf L$ be the cokernel of a Moore matrix $M_{a,x}:S(-2)^{3}\lto S(-1)^{3}$ for $a\in K^{3}$
representing a point $\bfa\in E$ with $a_{0}a_{1}a_{2}\neq 0$ and set 
$b=(a_{0}(a_{2}^{3}-a_{1}^{3}), a_{2}(a_{1}^{3}-a_{0}^{3}), a_{1}(a_{0}^{3}-a_{2}^{3}))$ as before. 

The one-dimensional vector space
$\Extgr^{1}_{R}(\mathbf L,\mathbf L)$ over $K$ can be realized as
\begin{align*}
\Extgr^{1}_{R}(\mathbf L,\mathbf L)\cong 
\frac{\left\{M_{b,y}\in \Mat_{3\times 3}(S_{1})\right\}}
{\left\{M_{b,y}\in \Mat_{3\times 3}(S_{1})\right\}\cap\left\{UM_{a,x}-M_{a,x}V\mid U,V\in  \Mat_{3\times 3}(K)\right\}}
\end{align*}
and the divergence $M_{b,y}\mapsto \mathsf{div}(M_{b,y})\in K$ induces an isomorphism
\begin{align*}
\Extgr^{1}_{R}(\mathbf L,\mathbf L)\xto[\cong]{\ \mathsf{div}\ } K\,. 
\end{align*}
\end{theorem}

\begin{proof}
As mentioned before, with $B=M_{a,x}^{\adj}$, the function $\tr(B\cdot-):Mat_{3\times 3}(S)\to S$ is $S$--linear,
whence each matrix $M_{b,y}$ satisfies $\tr(B\cdot M_{b,y})=0$ as we know this for the matrices
$M_{0}=M_{b,(1,0,0)}, M_{1}=M_{b,(0,1,0)}, M_{2}= M_{b,(0,0,1)}$ from above. Thus, the vector space
$\left\{M_{b,y}\in \Mat_{3\times 3}(S_{1})\right\}$ is contained in the numerator of the description of
$\Extgr^{1}_{R}(\mathbf L,\mathbf L)$ in (\ref{ext1})(\ref{ext1*}). As we know from Lemma
\ref{dimensions} that this extension group is one-dimensional, and, say, $\mathsf{div} M_{b,\iota(x)}=1\in K$
as noted above, it remains only to show that the denominator in the description here lies in the kernel of the divergence.

To this end, assume $M_{b,y} = UM_{a,x}-M_{a,x}V$ for some linear forms $y_{i}$ and some 
$U,V\in\Mat_{3\times 3}(K)$. Differentiating both sides with respect to $x_{0}$ and comparing the diagonal entries
yields the system of equations
\begin{align*}
a_{2i}(a_{2i-1}^{3}-a_{2i+1}^{3})\frac{\partial y_{0}}{\partial x_{0}} &= a_{2i}(u_{ii}-v_{ii})\,,\quad
\text{for $i\in\ZZ/3\ZZ$.}
\end{align*}
Dividing by $a_{2i}$, which is not zero by assumption, and then adding up shows that necessarily
$\sum_{i\in\ZZ/3\ZZ}u_{ii} = \sum_{i\in\ZZ/3\ZZ}v_{ii}$. Differentiating as well with respect to $x_{1}, x_{2}$,
comparing entries on both sides of the matrix equation and eliminating common factors of the form $a_{i}$
leads to the system of equations 
\begin{align*}
(a_{2}^{3}-a_{1}^{3})\frac{\partial y_{0}}{\partial x_{0}} &= u_{00}-v_{00}\,,\quad 
(a_{0}^{3}-a_{2}^{3}) \frac{\partial y_{2}}{\partial x_{2}}= u_{00}-v_{11}\,,\quad 
(a_{1}^{3}-a_{0}^{3})\frac{\partial y_{1}}{\partial x_{1}} = u_{00}-v_{22}\,,\\
(a_{0}^{3}-a_{2}^{3})\frac{\partial y_{1}}{\partial x_{1}} &= u_{11}-v_{00}\,,\quad 
(a_{1}^{3}-a_{0}^{3})\frac{\partial y_{0}}{\partial x_{0}} = u_{11}-v_{11}\,,\quad
(a_{2}^{3}-a_{1}^{3}) \frac{\partial y_{2}}{\partial x_{2}} = u_{11}-v_{22}\,,\\
(a_{1}^{3}-a_{0}^{3}) \frac{\partial y_{2}}{\partial x_{2}}&= u_{22}-v_{00}\,,\quad 
(a_{2}^{3}-a_{1}^{3})\frac{\partial y_{1}}{\partial x_{1}} = u_{22}-v_{11}\,,\quad 
(a_{0}^{3}-a_{2}^{3})\frac{\partial y_{0}}{\partial x_{0}} = u_{22}-v_{22}\,.
\end{align*}
Now at least one of the terms $(a_{i}^{3}-a_{i-1}^{3}), i\in \ZZ/3\ZZ,$ that occur as coefficients on the left-hand sides is nonzero,
as not all entries of $b$ are zero. Picking one such nonzero term and using it to solve for 
$\frac{\partial y_{i}}{\partial x_{i}}, i=0,1,2$, shows immediately that 
$\mathsf{div}(M_{b,y}) =
\frac{\partial y_{0}}{\partial x_{0}} + \frac{\partial y_{1}}{\partial x_{1}} + \frac{\partial y_{0}}{\partial x_{1}} = 0$
is a necessary condition on $M_{b,y}$ to be representable as $UM_{a,x}-M_{a,x}V$. In fact, we also know
that this condition is sufficient.
\end{proof}

\subsection*{The Proof of Theorem B}\quad\\
With $F$ as in Theorem B(\ref{thmba}), Atiyah's result Theorem \ref{Atiyah} shows that $F$ can be obtained as
an extension of a line bundle $L$ by itself, with $\deg L=0$ and $H^{0}(E,L)=0$. Applying $\Gamma_{*}$ to such
extension results in an extension of $\mathbf L$ by itself and those were classified above to yield a presentation
of $F$ as claimed.

Part (\ref{thmbb}) of Theorem B follows as the cokernel $F$ of the triangular block matrix fits into a short exact
sequence $0\to L\to F\to L\to 0$ with $L=\Coker M_{a,x}$, which yields immediately that $F$ is a vectorbundle of 
rank $2$ and degree $0$ that has no nonzero sections. Moreover, $F$ is indecomposable as the extension is not 
split, due to $\mathsf{div}(M_{a,x})=3\neq 0$ in $K$.

Part (\ref{thmbc}) of Theorem B follows from Atiyah's result and from Theorem A, as the line bundle $L$
in the short sequence above is uniquely determined by $F$ up to isomorphism.
\qed

\end{document}